\documentclass[a4paper,11pt,reqno]{amsart}
\usepackage{amsthm}
\usepackage{amsmath}
\usepackage{enumerate}
\usepackage{amsfonts}
\usepackage{amssymb}
\usepackage{fullpage}
\usepackage{amsmath,amscd}
\usepackage{stmaryrd}
\usepackage{graphicx}
\usepackage{array}
\usepackage{amsmath,amssymb,amsfonts,dsfont}
\usepackage[utf8]{inputenc} 
\usepackage[english]{babel}
\usepackage[T1]{fontenc} 
\usepackage{graphicx}
\usepackage{float}
\usepackage{pdfpages}
\usepackage[a4paper, margin = 3cm, bottom = 3cm]{geometry}
\usepackage{ifpdf}
\usepackage{marginnote}
\usepackage{hyperref}	
\hypersetup{pdfborder=0 0 0, 
	    colorlinks=true,
	    citecolor=black,
	    linkcolor=blue,
	    urlcolor=red,
	    pdfauthor={Guillaume Tahar}
	   }
\newtheorem{thm}{Theorem}[section]
\newtheorem{cor}[thm]{Corollary}

\newtheorem{lem}[thm]{Lemma}
\theoremstyle{definition}
\newtheorem{defn}[thm]{Definition}
\theoremstyle{remark}

\theoremstyle{definition}

\theoremstyle{definition}
\newtheorem{ex}[thm]{Example}
\theoremstyle{definition}

\numberwithin{equation}{section} 
\title{Ordered algebraic structures and classification of Semifields}
\author{Guillaume Tahar} 
\address[Guillaume Tahar]{Institut de Math{\'e}matiques de Jussieu - UMR CNRS 7586}
\email{guillaume.tahar@imj-prg.fr}
\date{July 18, 2017}
\keywords{Fields, Semifields, Idempotent structures}
\begin{document}
\begin{abstract}
Semifields are semirings in which every nonzero element has a multiplicative inverse. A rough classification uses the characteristic of the semifield, that is the isomorphism type of the semifield generated by the two neutral elements. For every characteristic, we provide a structure theorem that reduces the classification of semifields to the classification of better-known algebraic structures. Every semifield of characteristic $p$ is actually a field. There is an equivalence between semifields of characteristic one and lattice-ordered groups. Strict semifields of characteristic zero are quotients of cancellative semifields and there is an equivalence between concellative strict semifields and a particular class of partially ordered rings.\newline
\end{abstract}
\maketitle
\setcounter{tocdepth}{1}
\tableofcontents

\section{Introduction}

Recent trends such as tropical geometry tend to consider idempotent algebraic structures as natural objects of algebraic geometry. Semirings and semifields have a great importance in the Connes-Consani generalization of algebraic geometry, see \cite{CC} and \cite{L}.\newline
In this short note, we provide a systematic correspondance between semifields of a given characteristic and better known algebraic structures. In this way, we are certain that the examples of semifields we already know cover all possible phenomena.

\begin{defn}
A \textbf{semifield} is a set $F$ endowed with two associative laws $+$ an $\times$ satisfying the following conditions:\newline
(i) $(F,+)$ is a commutative monoid with a neutral element $0$.\newline
(ii) $(F\setminus\lbrace0\rbrace, \times)$ is a group.\newline
(iii) The law $\times$ is both left and right distributive with respect to $+$ and for every element $a$, we have $0.a=0$ and $a.0=0$.
\end{defn}

A \textbf{strict semifield} is a semifield that is not a field.

\begin{lem}
The multiplicative group of a strict semifield is torsion-free.
\end{lem}

\begin{proof}
We suppose an element $a$ of order $d>1$ in a semifield $F$. We have $a \sum_{i=0}^{d-1} a^{i} = \sum_{i=0}^{d-1} a^{i}$ so $1+\sum_{i=1}^{d-1} a^{i}=0$. Therefore, there is an additive inverse of $1$ and $F$ is a field.
\end{proof}

Enlarging the category of fields to semifields adds only a unique finite semifield that is the Boolean semifield.

\begin{cor}
A finite semifield is either a finite field or the Boolean semifield $B$.
\end{cor}

\begin{proof}
The multiplicative group of a finite strict semifield is finite and torsionfree so it is trivial. There are two elements in the semifield: $0$ and $1$. We have $1+1=1$ because otherwise it would be a field. Therefore, we get the Boolean semifield.
\end{proof}

In the following, we find that semifields are divided according to their characteristic. Semifields of characteristic $p$ are true fields so there is no need for further investigation.

\begin{thm}
For every semifield $F$, the characteristic semifield $Char(F)$ generated by the two neutral elements is one of the following isomorphism type:\newline
(i) $Char(F)$ is the Boolean semifield $B$. $F$ is of characteristic one.\newline
(ii) $Char(F)$ is one of the $\mathbb{F}_{p}$. $F$ is of characteristic $p$ with $p$ prime. Moreover, $F$ is automatically a field of characteristic $p$.\newline
(iii) $Char(F)$ is isomorphic to the positive rationals $\mathbb{Q}^{+}$. $F$ is of characteristic zero.
\end{thm}

\begin{proof}
If $Char(F)$ is finite, following Lemma 1.2, either $Char(F)$ is the Boolean semifield or $Char(F)$ is a finite field. In this case, there is an additive inverse of $1$ in $F$ and thus $F$ is a field. $Char(F)$ is the prime field of $F$.\newline
If $Char(F)$ is infinite, then $Char(F)$ includes every element of the form $1+\dots+1$ and they are all different from each other. Consequently, $Char(F)$ is isomorphic to $\mathbb{Q}^{+}$.
\end{proof}

\section{Semifields of characteristic one}

Semifields of characteristic one are exactly the Max-Plus algebras of partially ordered groups. There are no other examples.

\begin{thm}
A nonzero element $x$ of a semifield of characteristic one $F$ is said to be positive if we have $1+x^{-1}=1$. Positive elements $F^{+}$ define a lattice on the multiplicative group of $F$. Reciprocally, the union of an absorbing zero and any lattice-order group is a semifield whose addition is the supremum for the order and whose multiplication the group law. This semifield is of characteristic one.
\end{thm}

\begin{proof}
We check that the set of positive elements satisfy the four classical conditions to define a partially ordered structure on a group.\newline
First, $1 \in F^{+}$ because $1+1=1$.\newline
If we have $x,x^{-1} \in F^{+}$, then $x=1$. Indeed, if we have $1+x=1$ and $1+x^{-1}=1$, then we have $x+1=x$ and $x=1$.\newline
If $x \in F^{+}$, then for every nontrivial element $a$, we have $a^{-1}.x.a \in F^{+}$ because $1+a^{-1}.x^{-1}.a=a^{-1}(a+x^{-1}.a)=a^{-1}(1+x^{-1})a=1$.\newline
If we have $x,y \in F^{+}$, then we have $x+y \in F^{+}$ because $1+\dfrac{1}{x+y}=\dfrac{1+x+y}{x+y}=\frac{x+y}{x+y}=1$.\newline
Therefore, $F^{+}$ defines a partial order on the multiplicative group of $F$.\newline
We suppose $a \in F$ such that $a\geq x$ and $a \geq y$. Then, we have $a+a\geq x+y$ that is $a \geq x+y$. Therefore, $x+y$ is the least upper bound of $x$ and $y$. The multiplicative is lattice-ordered.\newline
We suppose $(G,\times,\geq)$ is a lattice-ordered group. Then we define on $\lbrace0\rbrace \cup G$ a semifield structure. The multiplication is the group law and $0$ is absorbing. The addition of two elements is the least upper bound (using the lattice structure) thus it is associative. The addition is idempotent so the semifield is of characteristic one. We have to check that the law $\times$ is distributive with respect to $+$. This comes from the fact that the order structure is translation-invariant in a group.
\end{proof}

We give two classical examples of semifields of characteristic one.

\begin{ex}
The Boolean semifield corresponds to the trivial lattice-order group.
The tropical semifield $(\mathbb{Z}\cup\lbrace-\infty\rbrace,max,+)$ corresponds to $(\mathbb{Z},+,\geq)$.
\end{ex}

\section{Semifields of characteristic zero}

In the following, we assume that semifields are strict semifields of characteristic zero. We leave aside fields of characteristic zero.\newline

\subsection{Cancellative semifields}

A remarkable class of strict semifields of characteristic one is given by the positive subsemifields of ordered fields. Among them there are $\mathbb{Q}^{+}$ or $\mathbb{R}^{+}$. In order to enlarge the scope, we consider strict semifields that keep a great similarity with those examples.

\begin{defn}
A semifield $F$ is cancellative if its underlying semigroup is cancellative, that is $x+z=y+z$ implies $x=y$ for every $x,y,z \in F$
\end{defn}

\begin{thm}
Every cancellative strict semifield $F$ of characteristic zero is embedded in a partially ordered ring $A_{F}$ such that the set of nonzero positive elements is closed under inversion (they are all invertible) and generate the whole ring (every element is a difference between two nonzero positive elements). Reciprocally, in any such ring, the set of positive elements is a cancellative strict semifield of characteristic zero.
\end{thm}

\begin{proof}
A cancellative semigroup is embedded in its Grothendieck group $A$. The multiplication of $F$ induces a multiplication on $A$ so $A$ inherits a ring structure. The semifield $F$ defines in $A$ a partial order if we consider $F$ as the subset of positive elements of $A$. $F$ is closed under addition and multiplication and its unique element whose opposite is also in $F$ is $0$. Besides, every nonzero element of $F$ is inversible and $F$ generates $A$.\newline
Reciprocally, we consider a partially ordered ring $A$ whose set of nonzero positive elements are invertible and generates $A$. The union of the zero and this set of stricly positive elements is closed under addition, multiplication and inversion. Therefore, it is a semifield $F$. Since $F$ embeds in a ring, it is cancellative and thus is of characteristic zero.
\end{proof}

This class of semifields obviously includes the positive subsemifields of ordered fields but there are far more general objects. Indeed, there is no product in the category of fields whereas we can define a product of rings.

\begin{ex}
The partially ordered ring $\mathbb{Q}^{2}$ with term to term addition and multiplication defines a cancellative semifield of characteristic zero. In a same way, we can construct semifields of the form $\lbrace 0 \rbrace \cup \mathcal{F}(X,\mathbb{R}^{\ast}_{+})$ starting from the ring of functions $\mathcal{F}(X,\mathbb{R})$.\newline
Another example is the semifield of rational fractions in one variable with strictly positive coefficients (plus an absorbing zero).
\end{ex}

Though there are no nontrivial nilpotent elements in a semifield, the partially ordered ring generated by a semifield may fail to be reduced.

\begin{ex}
The ring of dual numbers $\mathbb{R}\left[ \epsilon \right] $ with $\epsilon^{2}=0$ is generated by the cancellative semifield $\lbrace 0 \rbrace \cup  \lbrace a+b.\epsilon, a \in \mathbb{R}^{\ast}_{+}, b \in \mathbb{R}\rbrace$. The partial order on $\mathbb{R}\left[ \epsilon \right]$ is given by the total order on the first coordinate.
\end{ex}

\subsection{Noncancellative semifields}

There is a well-defined notion of quotient in the category of semifields of characteristic zero. In the following, we prove that however complicated a noncancellative semifield may be, it is a quotient of a cancellative semifield.

\begin{thm}
Every strict semifield of characteristic zero is a quotient of a cancellative strict semifield of characteristic zero.
\end{thm}

\begin{proof}
We consider the semigroup $L_{F}$ of formal sums of elements of the semifield $F$ modulo the relations $[x]+[\theta.x]=[(\theta+1)x]$ when $\theta$ is a rational number. The multiplicative law of $F$ induces a distributive multiplicative law on $L(F)$. Since the additive law of $L_{F}$ is given by the addition of formal sums, it is cancellative. Then, we consider $Q(F)$ formed by fractions of elements of $L_{F}$. $Q(F)$ is a cancellative strict semifield of characteristic zero.\newline
The projection morphism $\phi$ from $Q(F)$ to $F$ is defined by $\phi([x])=x$ an $\phi([x]+[y])=x+y$.
\end{proof}

There is a wide variety of noncancellative semifields but they always can be constructed as quotients of cancellative semifields.

\begin{ex}
The semifield $\lbrace0\rbrace \cup \cup_{k \in \mathbb{Z}} \lbrace qX^{k}, q \in \mathbb{Q}^{\ast}_{+}\rbrace$ with the usual multiplication and such that $qX^{k}+q'X^{k'}=(q+q')X^{k}$ when $k=k'$ and $qX^{k}+q'X^{k'}=qX^{k}$ when $k>k'$ is not cancellative. However, it can be interpreted as the semifield of the asymptotic equivalence classes of the semifield of rational fonctions in one variable with positive rational coefficients. The latter semifield is cancellative.
\end{ex}

\nopagebreak
\vskip.5cm
\end{document}